\documentclass[10pt,reqno,letter]{amsart}

\title[Analyticity of solutions to the Euler Equations in a Half Space]{The Domain of Analyticity of Solutions to the Three-Dimensional Euler Equations in a Half Space}
\author{Igor Kukavica}
\author{Vlad Vicol}

\address{Department of Mathematics, University of Southern California, Los Angeles, CA 90089}
\email{kukavica@usc.edu}

\address{Department of Mathematics, University of Southern California, Los Angeles, CA 90089}
\email{vicol@usc.edu}

\newtheorem{theorem}{Theorem}[section]

\newtheorem{lemma}[theorem]{Lemma}

\theoremstyle{definition}

\newtheorem{remark}{Remark}

\def\CCC{\mathcal{C}}
\def\PPP{\mathcal{P}}
\def\hhh{\Omega}

\def\Xtau{X_{\tau(t)}}
\def\Ytau{Y_{\tau(t)}}

\def\curl{\mathop{\rm curl} \nolimits}

\subjclass[2000]{Primary: 76B03; Secondary: 35L60.}

\begin{document}
\maketitle

%The abstract of your paper
\begin{abstract}
We address the problem of analyticity up to the boundary of solutions to the Euler equations in the half space. We characterize the rate of decay of the real-analyticity radius of the solution $u(t)$ in terms of $\exp{\int_{0}^{t} \Vert \nabla u(s) \Vert_{L^\infty} ds}$, improving the previously known results. We also prove the persistence of the sub-analytic Gevrey-class regularity for the Euler equations in a half space, and obtain an explicit rate of decay of the radius of Gevrey-class regularity.
\end{abstract}

\section{Introduction}\label{sec:intro}

The Euler equations on a half space for the velocity vector field
$u(x,t)$ and the scalar pressure field $p(x,t)$, where
$x\in \hhh = \{x\in {\mathbb R}^3:x_3 >0\}$ and $t \geq 0$, are
given by
\begin{align}
&\partial_t u + (u\cdot \nabla) u + \nabla p = 0, \  \mbox{in}\ \hhh \times (0,\infty), \tag{E.1}\label{eq:E1}\\
&\nabla \cdot u = 0, \  \mbox{in}\ \hhh \times (0,\infty),  \tag{E.2}\label{eq:E2}\\
& u \cdot n = 0, \  \mbox{on}\ \partial\hhh \times (0,\infty)
\tag{E.3}\label{eq:E3},
\end{align}
where $n = (0,0,-1)$ is the outward unit normal to $\partial\hhh =
\{ x\in {\mathbb R}^3:x_3 =0\}$. We consider the initial value problem associated to
\eqref{eq:E1}--\eqref{eq:E3} with a divergence free initial datum
\begin{align}u(0) = u_0,\tag{E.4}\  \mbox{in}\ \hhh. \label{eq:E4}\end{align} The local existence and uniqueness of $H^r$-solutions, with $r > 3/2 + 1$, on a maximal time interval $[0,T_*)$ holds (cf.~\cite{BoB,EM,Ka,MB,T}), and $\lim_{T\nearrow T_*} \int_{0}^{T} \Vert \curl u(t) \Vert_{L^\infty} dt = \infty$, if $T_*<\infty$ (cf.~\cite{BKM}); additionally the persistence of $C^\infty$ smoothness was proven by Foias, Frisch, and Temam \cite{FFT}. In this paper we address the solutions of the Euler initial value problem evolving from real-analytic and Gevrey-class initial datum (up to the boundary), and characterize the domain of analyticity. We emphasize that the radius of real-analyticity gives an estimate on the minimal scale in the flow \cite{HKR,K2}, and it also gives the explicit rate of exponential decay of its Fourier coefficients \cite{FT}.

In a three dimensional bounded domain, the persistence of analyticity was proven by Bardos and Benachour \cite{BB} by an implicit argument (see also Alinhac and M\'etivier \cite{AM}). In \cite{B,Be} the authors give an explicit estimate on the radius of analyticity, but which vanishes in finite time (independent of $T_*$). However, the proof of persistency \cite{BB} can be modified to show that the radius of analyticity decays at a rate proportional to the exponential of a high Sobolev norm of the solution (see also \cite{AM}). On the three dimensional periodic domain (or equivalently on ${\mathbb R}^3$) this is the same rate obtained by Levermore and Oliver in \cite{LO}, using the method of Gevrey-class regularity. This Fourier based method was introduced by Foias and Temam \cite{FT} to study the analyticity of the Navier-Stokes equations. For further results on analyticity see \cite{AM,CTV,FTi,GK1,GK2,K1,K2,L1,L2,Lb,OT,SC}. Explicit and even algebraic lower bounds for the radius of analyticity for dispersive equations were obtained by Bona, Gruji\'c, and Kalisch in \cite{BGK1,BGK2} (see also \cite{BG,BL}).

In \cite{KV} we have proven that in the periodic setting, or on ${\mathbb R}^3$, the analyticity radius decays algebraically in the Sobolev norm $\Vert \curl u(t) \Vert_{H^r}$, with $r>7/2$, and exponentially in $\int_{0}^{t} \Vert \nabla u(s) \Vert_{L^\infty}$, for all $t<T_*$. In the present paper we show that the algebraic dependence on the Sobolev norm holds in the case when  the domain has boundaries (cf.~Theorem \ref{thm:main}), thereby improving the previously known results. The interior analyticity in the case of the half-space, for short time (independent of $T_*$), was treated in \cite{SC}. We note that the shear flow example of Bardos and Titi \cite{BT} (cf.~\cite{DM,Y2}) may be used to construct explicit solutions to the three-dimensional Euler equations whose radius of analyticity is decaying for all time.
%Therefore it is important to give sharp lower bounds for the rate of the decay of the radius for generic solutions of the Euler equations.

Additionally we prove the persistence of sub-analytic Gevrey-class regularity up to the boundary (cf.~\cite{FT,LM}) for the Euler equations on the half space. To the best of our knowledge this was only known for the periodic domain cf.~\cite{KV,LO}, but not for a domain with boundary. The methods of \cite{AM,BB,Lb,SC} rely essentially on the special structure of the complex holomorphic functions, and do not apply to the non-analytic Gevrey-class setting.

The presence of the boundary creates several difficulties that do not arise in the periodic setting. In particular we cannot use Fourier-based methods, nor can we use the vorticity formulation of the equations. Instead we need to estimate the pressure, which satisfies (cf.~\cite{T}) the elliptic Neumann problem
\begin{align}
  &- \Delta p = \partial_j u_i \partial_i u_j,\  \mbox{in}\ \hhh \times (0,\infty), \tag{P.1}
  \label{eq:P1}\\
  &\frac{\partial p}{\partial n} = (u\cdot \nabla) u \cdot n = 0, \  \mbox{on}\ \partial\hhh \times (0,\infty), \tag{P.2} \label{eq:P20}
\end{align}
since $n=(0,0,-1)$, where the summation convention on repeated indices was used in \eqref{eq:P1}. In order to close our argument we need to show that the pressure has the same analyticity radius as the velocity, and so we cannot appeal to the inductive argument of Lions and Magenes \cite{LM}. Moreover, the nature of the elliptic/hyperbolic boundary value problem imposes certain restrictions on the weights of the Sobolev norms that comprise the analytic norm. The analytic norm we define (cf.~Section \ref{sec:proof}) respects the symmetries of the problem and is adequate to account for the transfer of derivatives arising in the higher regularity estimates for the pressure.

The paper is organized as follows. In Section~\ref{sec:thm} we state our main result, Theorem~\ref{thm:main}. In Section~\ref{sec:proof} we prove the main theorem assuming two key estimates on the convection term and the pressure term, Lemma~\ref{lemma:commutator} and Lemma~\ref{lemma:pressure}. Section~\ref{sec:commutator} contains the proof of the commutator estimate Lemma~\ref{lemma:commutator}, and lastly, the higher regularity estimates for the pressure and the proof of Lemma~\ref{lemma:pressure} are given in Section~\ref{sec:pressure}.

\section{Main Theorems}\label{sec:thm}
The following statement is our main theorem addressing the analyticity of the solution. Theorem~\ref{thm:gevrey} below concerns the Gevrey class persistence.
\begin{theorem}\label{thm:main}
Fix $r>9/2$. Let $u_0 \in H^r(\Omega)$ be divergence-free and
uniformly real-analytic in $\Omega$. Then the unique solution $u(t) \in C(0,T_*;H^r(\Omega))$ of the initial value problem associated to the Euler equations \eqref{eq:E1}--\eqref{eq:E4} is real-analytic for all time $t<T_*$, where $T_* \in (0,\infty]$ denotes the maximal time of existence of the $H^r$-solution. Moreover, the uniform
radius of space analyticity $\tau(t)$ of $u(t)$ satisfies
  \begin{align}
    \tau(t) \geq \frac{1}{C_0(1+t)} \exp\left({-C\int_{0}^{t}
    \Vert{\nabla u(s)}\Vert_{L^\infty} ds}\right),\label{eq:thm}
  \end{align}
where $C>0$ is a constant that depends only on $r$, while $C_0$ has additional dependence on $u_0$ as described in \eqref{eq:C0} below.
\end{theorem}
\begin{remark}
  The lower bound \eqref{eq:thm} improves the rate of decay from Bardos and Benachour \cite{BB} on a bounded domain (which can be inferred to be proportional to $\exp{\int_{0}^{t} \Vert u(s) \Vert_{H^r} ds}$), and it matches the rate of decay we obtained in \cite{KV} on the periodic domain.
\end{remark}
\begin{remark}
  The proof of Theorem~\ref{thm:main} also works in the case of the half-plane (recall that in two dimensions $T_*$ may be taken arbitrarily large, cf.~\cite{MB,Y1}) with the same lower bound on the radius of analyticity of the solution. Since in two dimensions $\Vert \nabla u(t) \Vert_{L^\infty}$ grows at a rate of $C \exp(Ct)$, for some positive constant $C$, the estimate \eqref{eq:thm} shows that the rate of decay of the analyticity radius is at least $C \exp(- C \exp(Ct))$, for some $C>0$. This recovers the two-dimensional rate of decay obtained by Bardos, Benachour and Zerner \cite{BBZ} on a bounded domain and by the authors of this paper on the torus \cite{KV}. It would be interesting if one could prove a similar lower bound to \eqref{eq:thm} but where the quantity $\int_{0}^{t} \Vert \nabla u(s) \Vert_{L^\infty}\, ds$ is replaced by $\int_{0}^{t} \Vert \curl u(s) \Vert_{L^\infty}\, ds$. In particular, such an estimate would imply in two dimensions that the radius of analyticity decays as a single exponential in time.
\end{remark}
Recall (cf.~\cite{LM}) that a smooth function $v$ is uniformly of Gevrey-class $s$, with $s\geq 1$, if there exist $M,\tau >0$ such that
\begin{align}\label{eq:gevrey}
  |\partial^\alpha v(x)| \leq M
      \frac{|\alpha|!^s}{\tau^{|\alpha|}},
\end{align} for all $x\in \hhh$ and all multi-indices $\alpha \in {\mathbb N}_{0}^3$. When $s=1$ we recover the class of real-analytic functions, and for $s\in(1,\infty)$ these functions are $C^\infty$ smooth but might not be analytic. We call the constant $\tau$ in \eqref{eq:gevrey} the radius of Gevrey-class regularity. The following theorem shows the persistence of the Gevrey-class regularity for the Euler equations in a half-space.
\begin{theorem}
\label{thm:gevrey}
Fix $r>9/2$. Let $u_0$ be uniformly of Gevrey-class $s$ on $\hhh$, with $s>1$, and divergence-free. Then the unique $H^r$-solution $u(t)$ of the initial value problem \eqref{eq:E1}--\eqref{eq:E4} on $[0,T_*)$ is of Gevrey-class $s$, for all $t<T_*$, and the radius $\tau(t)$ of Gevrey-class regularity of the solution satisfies the lower bound \eqref{eq:thm}.
\end{theorem}

\section{Proofs of Theorem~\ref{thm:main} and Theorem~\ref{thm:gevrey}} \label{sec:proof}
For a multi-index $\alpha = (\alpha_1,\alpha_2,\alpha_3)$ in ${\mathbb
N}_{0}^{3}$, we denote $\alpha' = (\alpha_1,\alpha_2)$. Define the
Sobolev and Lipshitz semi-norms $|\cdot|_m$ and $|\cdot|_{m,\infty}$ by
\begin{align}
  |v|_m =\sum\limits_{|\alpha| = m}
  M_\alpha
  \Vert{\partial^\alpha v}\Vert_{L^2}, \label{eq:m}
\end{align}
and
\begin{align*}
  |v|_{m,\infty} = \sum\limits_{|\alpha|=m} M_\alpha
  \Vert{\partial^\alpha v}\Vert_{L^\infty},
\end{align*}
where
\begin{align}
M_\alpha = \frac{|\alpha'|!}{\alpha'!} = {\alpha_1 + \alpha_2
\choose \alpha_1}. \label{eq:Mdef}
\end{align}
The need for the binomial weights $M_\alpha$ in \eqref{eq:m} shall be evident in Section~\ref{sec:pressure} where we study the higher
regularity estimates associated with the Neumann problem \eqref{eq:P1}--\eqref{eq:P20}. For $s
\geq 1$ and $\tau>0$, define the space
\begin{align*}X_\tau = \{ v \in
C^\infty(\hhh):\Vert{v}\Vert_{X_\tau} < \infty\},
\end{align*} where
\begin{align*}
  \Vert{v}\Vert_{X_\tau} = \sum\limits_{m=3}^{\infty}
  |v|_m \frac{\tau^{m-3}}{(m-3)!^{s}}.
\end{align*}
Similarly let $Y_\tau = \{ v \in C^\infty(\hhh):\Vert
v\Vert_{Y_\tau} < \infty\}$, where
\begin{align*}
  \Vert{v}\Vert_{Y_\tau} = \sum\limits_{m=4}^{\infty}
  |v|_m \frac{(m-3)\tau^{m-4}}{(m-3)!^{s}} .
\end{align*}
\begin{remark}
 The above defined spaces $X_\tau$ and $Y_\tau$ can be identified with the classical Gevrey-$s$ classes as defined in \cite{LM}. On the full space or on the torus, the Gevrey-$s$ classes can also be identified with ${\mathcal D}((-\Delta)^{r/2} \exp{(\tau (-\Delta)^{1/2s})})$ (cf.~\cite{FT,KV,LO}).
\end{remark}

We shall prove Theorems~\ref{thm:main} and \ref{thm:gevrey} simultaneously by looking at the evolution equation in Gevrey-$s$ classes with $s\geq1$. If $u_0$ is of Gevrey-class $s$ in $\Omega$, with $s\geq1$, then there exists $\tau(0)>0$ such that $u_0 \in X_{\tau(0)}$, and moreover $\tau(0)$ can be chosen arbitrarily close to the uniform real-analyticity
radius of $u_0$, respectively to the radius of Gevrey-class regularity. Let $u(t)$ be the classical $H^r$-solution of the
initial value problem \eqref{eq:E1}--\eqref{eq:E4}.

With the notations of Section~\ref{sec:thm} we have an {\it a
priori} estimate
\begin{align}
  \frac{d}{dt} \Vert{u(t)}\Vert_{\Xtau} =
  \dot{\tau}(t) \Vert{u(t)}\Vert_{\Ytau} + \sum\limits_{m=3}^{\infty}
  \left(\frac{d}{dt} |u(t)|_m\right)
  \frac{\tau(t)^{m-3}}{(m-3)!^{s}}. \label{eq:ODE1}
\end{align}
Fix $m\geq 3$. In order to estimate $(d/dt) |u(t)|_m$, for each
$|\alpha|=m$ we apply $\partial^\alpha$ on \eqref{eq:E1} and take
the $L^2$-inner product with $\partial^\alpha u$. We obtain
\begin{align}
  \frac12 \frac{d}{dt} \Vert{\partial^\alpha
  u}\Vert_{L^2}^2 + < \partial^\alpha (u \cdot \nabla u),
  \partial^\alpha u> + < \nabla \partial^\alpha p, \partial^\alpha
  u> = 0. \label{eq:inner1}
\end{align}
On the second term on the left, we apply the Leibniz rule and recall
that $< u \cdot \nabla \partial^\alpha u,
\partial^\alpha u> = 0$. For the third term on the left of \eqref{eq:inner1} we note that since $n = (0,0,-1)$ and $u \cdot n = 0$ on
$\partial\hhh$, we have that $\partial^\alpha u \cdot n =0$ for all
$\alpha$ such that $\alpha_3 = 0$. Together with $\nabla \cdot u =
0$ in $\hhh$ this implies that $< \nabla
\partial^\alpha p,
\partial^\alpha u> = 0$ whenever $\alpha_3 = 0$. Using the
Cauchy-Schwarz inequality and summing over $|\alpha|=m$ we then
obtain
\begin{align*}
\frac{d}{dt} |u|_m \leq  \sum\limits_{|\alpha| =
m}\sum\limits_{\beta \leq \alpha, \beta\neq 0} M_\alpha {\alpha
\choose \beta} \Vert{\partial^\beta u \cdot \nabla
\partial^{\alpha-\beta} u}\Vert_{L^2} + \sum\limits_{|\alpha|
= m, \alpha_3\neq 0}M_\alpha \Vert{\nabla \partial^\alpha
p}\Vert_{L^2}.
\end{align*}
Combined with \eqref{eq:ODE1}, the above estimate shows that
\begin{align}
  \frac{d}{dt} \Vert{u(t)}\Vert_{\Xtau} \leq
  \dot{\tau}(t) \Vert{u(t)}\Vert_{\Ytau} + \CCC + \PPP
  \label{eq:ODE2},
\end{align}
where the upper bound on the commutator term is given by
\begin{align*}
\CCC = \sum\limits_{m=3}^{\infty} \sum\limits_{|\alpha| =
m}\sum\limits_{\beta \leq \alpha, \beta\neq 0} M_\alpha {\alpha
\choose \beta} \Vert{\partial^\beta u \cdot \nabla
\partial^{\alpha-\beta} u}\Vert_{L^2}
\frac{\tau^{m-3}}{(m-3)!^{s}},
\end{align*}
and the upper bound on the pressure term is
\begin{align*}
\PPP = \sum\limits_{m=3}^{\infty} \sum\limits_{|\alpha| = m,
\alpha_3\neq 0}M_\alpha \Vert{\nabla \partial^\alpha p}\Vert_{L^2}
\frac{\tau^{m-3}}{(m-3)!^{s}}.
\end{align*}
In order to estimate $\CCC$ we use the following lemma, the proof of
which is given in Section~\ref{sec:commutator} below.
\begin{lemma}\label{lemma:commutator}
There exists a sufficiently large constant $C>0$ such that
\begin{align*}
  \CCC \leq C \left(\mathcal{C}_1 +
\mathcal{C}_2 \Vert{u}\Vert_{Y_\tau}\right),
\end{align*}
where
\begin{align*}
  \mathcal{C}_1 = |u|_{1,\infty} |u|_3 + |u|_{2,\infty} |u|_2 + \tau |u|_{2,\infty} |u|_3,
\end{align*}
and
\begin{align*}
  \mathcal{C}_2 = \tau |u|_{1,\infty} + \tau^{2} |u|_{2,\infty} + \tau^3 |u|_{3,\infty} + \tau^{3/2}  \Vert{u}\Vert_{X_\tau}.
\end{align*}
\end{lemma}
The following lemma shall be used to estimate $\PPP$. The proof is
given in Section~\ref{sec:pressure} below.
\begin{lemma}\label{lemma:pressure}
There exists a sufficiently large constant $C>0$ such that
\begin{align*}
  \PPP  \leq C \left(\PPP_1 + \PPP_2
\Vert{u}\Vert_{Y_\tau}\right),
\end{align*}
where
\begin{align*}
  \PPP_1 = |u|_{1,\infty} |u|_3 +  |u|_{2,\infty} |u|_2 +  \tau |u|_{2,\infty} |u|_3 +  \tau^2 |u|_{3,\infty} |u|_3,
\end{align*}
and
\begin{align*}
  \PPP_2 = \tau |u|_{1,\infty} + \tau^2 |u|_{2,\infty} + \tau^3 |u|_{3,\infty} + \tau^{3/2} \Vert{u}\Vert_{X_\tau}.
\end{align*}
\end{lemma}
Let $r>9/2$ be fixed. The Sobolev embedding theorem, the two
lemmas above, and \eqref{eq:ODE2} imply
\begin{align}
  & \frac{d}{dt} \Vert{u(t)}\Vert_{X_{\tau(t)}} \leq
  \dot{\tau}(t) \Vert{u(t)}\Vert_{Y_{\tau(t)}} + C
  \Vert{u(t)}\Vert_{H^r}^2 ( 1+ \tau(t)^2) \notag\\
  &\ + C
  \Vert{u(t)}\Vert_{Y_{\tau(t)}} \left( \tau(t)
  \Vert{\nabla u(t)}\Vert_{L^\infty} + (\tau(t)^2 + \tau(t)^3)
  \Vert{u(t)}\Vert_{H^r} + \tau(t)^{3/2}
  \Vert{u(t)}\Vert_{X_{\tau(t)}}\right). \label{eq:ODE3}
\end{align}
If $\tau(t)$ decreases fast enough so that for all $0 \leq t < T_*$
we have
\begin{align} \label{eq:condition1}
\dot{\tau}(t) + C \tau(t)
  \Vert{\nabla u(t)}\Vert_{L^\infty} + C (\tau(t)^2 + \tau(t)^3)
  \Vert{u(t)}\Vert_{H^r} + C \tau(t)^{3/2}
  \Vert{u(t)}\Vert_{X_{\tau(t)}} \leq 0,
\end{align}
then \eqref{eq:ODE3} implies that
\begin{align*}
\frac{d}{dt} \Vert{u(t)}\Vert_{X_{\tau(t)}} \leq C
  \Vert{u(t)}\Vert_{H^r}^2 ( 1+ \tau(0)^2),
\end{align*}
and therefore
\begin{align*}
\Vert{u(t)}\Vert_{X_{\tau(t)}} \leq \Vert{u_0}\Vert_{X_{\tau(0)}} +
C_{\tau(0)} \int_{0}^{t} \Vert{u(s)}\Vert_{H^r}^2 ds = M(t),
\end{align*}
for all $0\leq t < T_*$, where $C_{\tau(0)} =1+
\tau(0)^2$. Since $\tau$ must be chosen to be a decreasing function,
a sufficient condition for \eqref{eq:condition1} to hold is that
\begin{align}
\dot{\tau}(t) + C \tau(t)
  \Vert{\nabla u(t)}\Vert_{L^\infty} + C \tau(t)^{3/2}
  \left(
  C_{\tau(0)}' \Vert{u(t)}\Vert_{H^r} +
  M(t)\right) \leq 0, \label{eq:condition2}
\end{align}
where $C_{\tau(0)}' = \tau(0)^{1/2} + \tau(0)^{3/2}$. For simplicity
of the exposition we denote
\begin{align*}
  G(t) = \exp\left( C \int_{0}^{t} \Vert{\nabla
  u(s)}\Vert_{L^\infty} ds\right),
\end{align*}
where the constant $C>0$ is taken sufficiently large so that
$\Vert{u(t)}\Vert_{H^r}^2 \leq \Vert{u_0}\Vert_{H^r}^2 G(t)$. It
then follows that \eqref{eq:condition2} is satisfied if we let
\begin{align*}
\tau(t) = G(t)^{-1/2} \left( \tau(0)^{-1/2} + C \int_{0}^{t} \left(
C_{\tau(0)}' \Vert{u(s)}\Vert_{H^r} + M(s)\right) G(s)^{-1}
ds\right)^{-1/2}.
\end{align*}
The lower bound \eqref{eq:thm} on the
radius of analyticity stated in Theorem~\ref{thm:main} is then
obtained by noting that
\begin{align}
&\tau(0)^{-1/2} + C \int_{0}^{t} \left( C_{\tau(0)}'
\Vert{u(s)}\Vert_{H^r} + M(s)\right) G(s)^{-1} ds \notag
\\
& \qquad \qquad \leq \tau(0)^{-1/2} + C \int_{0}^{t} \left(
C_{\tau(0)}' \Vert{u_0}\Vert_{H^r} + \Vert{u_0}\Vert_{X_{\tau(0)}} +
s C_{\tau(0)}
\Vert{u_0}\Vert_{H^r} ^2\right) ds\notag\\
& \qquad \qquad \leq C_0 (1+t)^2, \label{eq:C0}
\end{align}
and therefore
\begin{align*}
\tau(t) \geq G(t)^{-1/2} \frac{C_0}{1+t}.
\end{align*} The last inequality in \eqref{eq:C0} above gives the
explicit dependence of $C_0$ on $u_0$. This concludes the {\it a
priori} estimates that are used to prove Theorem~\ref{thm:main}. The proof can be made formal by considering an approximating solution $u^{(n)}$, $n\in {\mathbb N}$, proving the above estimates for $u^{(n)}$, and then taking the limit as $n\rightarrow \infty$. We omit these details.

\section{The commutator estimate} \label{sec:commutator}
Before we prove Lemma~\ref{lemma:commutator} we state and prove two
useful lemmas about multi-indexes, that will be used throughout in
Sections~\ref{sec:commutator} and \ref{sec:pressure} below.
\begin{lemma}\label{lemma:choose}
We have
\begin{align}
{\alpha \choose \beta} M_{\alpha}  M_{\beta}^{-1}
M_{\alpha-\beta}^{-1} \leq {|\alpha| \choose |\beta|}
 \label{eq:choose}
\end{align}
for all $\alpha,\beta \in \mathbb{N}_0^3$ with $\beta \leq \alpha$.
\end{lemma}
\begin{proof}
Using \eqref{eq:Mdef} we have that
\begin{align*}
  {\alpha' \choose \beta'} M_{\alpha}  M_{\beta}^{-1}
M_{\alpha-\beta}^{-1} = {|\alpha'| \choose |\beta'|},
\end{align*}
and hence the left side of \eqref{eq:choose} is bounded by
\begin{align*}
{|\alpha'| \choose |\beta'|} {\alpha_3 \choose \beta_3}.
\end{align*}The lemma then follows from
\begin{align*}
  {n \choose i} {m \choose j} \leq {n+m \choose i+j},
\end{align*}
for any $n,m\geq 0$ such that $n\geq i$ and $m\geq j$, which in turn we obtain by computing the coefficient in front of $x^{i+j}$ in the binomial
expansions of $(1+x)^n (1+x)^m$ and $(1+x)^{m+n}$.
\end{proof}
The second lemma allows us to re-write certain double sums involving
multi-indices.
\begin{lemma}\label{lemma:product}
Let $\{x_\lambda\}_{\lambda\in{\mathbb N}_{0}^{3}}$ and
$\{y_\lambda\}_{\lambda \in {\mathbb N}_{0}^{3}}$ be real
numbers. Then we have
\begin{align} \label{eq:product}
  \sum\limits_{|\alpha|=m} \sum\limits_{|\beta|=j, \beta\leq \alpha} x_\beta y_{\alpha
  - \beta} =  \left( \sum\limits_{|\beta| =j} x_\beta\right) \left(
  \sum\limits_{|\gamma| = m-j} y_\gamma\right).
\end{align}
\end{lemma}
The proof of the above lemma is omitted: it consists of
re-labeling of the terms on the left side of \eqref{eq:product}. Now
we proceed by proving the commutator estimate.
\begin{proof}[Proof of Lemma~\ref{lemma:commutator}]
We have
\begin{align*}
  \CCC = \sum\limits_{m=3}^{\infty} \sum\limits_{j=1}^{m} \CCC_{m,j},
\end{align*}
where we denoted
\begin{align}\label{eq:Cmj}
  \CCC_{m,j} =
  \frac{\tau^{m-3}}{(m-3)!^{s}} \sum\limits_{|\alpha|=m}
  \sum\limits_{|\beta|=j, \beta\leq \alpha} M_\alpha {\alpha \choose
  \beta} \Vert{\partial^\beta u \cdot \nabla \partial^{\alpha-\beta}
  u}\Vert_{L^2}.
\end{align}
We now split the right side of the above equality into seven terms according
to the values of $m$ and $j$, and prove the following estimates. For
low $j$, we claim
\begin{align}
&\sum\limits_{m=3}^{\infty} \CCC_{m,1}\leq C |u|_{1,\infty} |u|_3 +
C \tau |u|_{1,\infty} \Vert{u}\Vert_{Y_\tau}\label{eq:low1},\\
&\sum\limits_{m=3}^{\infty} \CCC_{m,2}\leq  C |u|_{2,\infty} |u|_2 +
C \tau |u|_{2,\infty} |u|_3+C \tau^2 |u|_{2,\infty}
\Vert{u}\Vert_{Y_\tau} \label{eq:low2}, \end{align} for intermediate
$j$, we have
\begin{align}
&\sum\limits_{m=6}^{\infty} \sum\limits_{j=3}^{[m/2]} \CCC_{m,j}
\leq C \tau^{3/2} \Vert{u}\Vert_{X_\tau} \Vert{u}\Vert_{Y_\tau}\label{eq:lowmed},\\
&\sum\limits_{m=7}^{\infty} \sum\limits_{j=[m/2]+1}^{m-3} \CCC_{m,j}
\leq C \tau^{3/2} \Vert{u}\Vert_{X_\tau}
\Vert{u}\Vert_{Y_\tau}\label{eq:highmed}, \end{align} and for high
$j$,
\begin{align}
&\sum\limits_{m=5}^{\infty} \CCC_{m,m-2} \leq C \tau^3 |u|_{3,\infty} \Vert{u}\Vert_{Y_\tau}\label{eq:highm-2},\\
&\sum\limits_{m=4}^{\infty} \CCC_{m,m-1} \leq C \tau |u|_{2,\infty}
|u|_3+C \tau^2 |u|_{2,\infty} \Vert{u}\Vert_{Y_\tau}\label{eq:highm-1},\\
&\sum\limits_{m=3}^{\infty} \CCC_{m,m} \leq  C |u|_{1,\infty} |u|_3
+ C \tau |u|_{1,\infty} \Vert{u}\Vert_{Y_\tau} \label{eq:highm}.
\end{align}
Due to symmetry we shall only prove
\eqref{eq:low1}--\eqref{eq:lowmed} and indicate the necessary
modifications for \eqref{eq:highmed}--\eqref{eq:highm}.

Proof of \eqref{eq:low1}: The H\"older inequality, \eqref{eq:Cmj},
and Lemma~\ref{lemma:choose} imply that
\begin{align} \label{eq:Cmjlow1-1}
\sum\limits_{m=3}^{\infty} \CCC_{m,1} &=
\sum\limits_{|\alpha|=3}\sum\limits_{|\beta|=1, \beta\leq \alpha}
 \left(M_\beta \Vert{\partial^\beta u}\Vert_{L^\infty} \right)
  \left( M_{\alpha - \beta} \Vert{\partial^{\alpha-\beta} \nabla u}\Vert_{L^2}\right) M_\alpha M_{\beta}^{-1} M_{\alpha-\beta}^{-1} {\alpha \choose \beta} \notag\\
& +\sum\limits_{m=4}^{\infty} \sum\limits_{|\alpha|=m}
  \sum\limits_{|\beta|=1, \beta\leq \alpha} \left( M_\beta \Vert{\partial^\beta
  u}\Vert_{L^\infty}\right) \left( M_{\alpha-\beta} \Vert{\partial^{\alpha-\beta} \nabla u}\Vert_{L^2} \frac{(m-3) \tau^{m-4}}{(m-3)!^s}\right) \notag \\
& \qquad \qquad \times M_\alpha M_{\beta}^{-1} M_{\alpha-\beta}^{-1}
{\alpha \choose \beta} \frac{1}{m-3}\ \tau \notag \\
& \leq C \sum\limits_{|\alpha|=3}\sum\limits_{|\beta|=1, \beta\leq
\alpha} \left(M_\beta \Vert{\partial^\beta u}\Vert_{L^\infty}
\right)
  \left( M_{\alpha - \beta} \Vert{\partial^{\alpha-\beta} \nabla u}\Vert_{L^2}\right) \notag\\
& + C \tau \sum\limits_{m=4}^{\infty} \sum\limits_{|\alpha|=m}
  \sum\limits_{|\beta|=1, \beta\leq \alpha} \left( M_\beta \Vert{\partial^\beta
  u}\Vert_{L^\infty}\right) \notag\\
& \qquad \qquad \times \left( M_{\alpha-\beta} \Vert{\partial^{\alpha-\beta} \nabla u}\Vert_{L^2} \frac{(m-3)
  \tau^{m-4}}{(m-3)!^s}\right) \frac{m}{m-3}.
\end{align}
The first sum on the far right side of \eqref{eq:Cmjlow1-1} can be
estimated by
\begin{align*}
  C |u|_{1,\infty} |\nabla u|_2
  \leq C  |u|_{1,\infty} |u|_3.
\end{align*}
Since $m \geq 4$, Lemma~\ref{lemma:product} implies that the second
term on the far right side of \eqref{eq:Cmjlow1-1} is bounded by
\begin{align*}
 C \tau |u|_{1,\infty} \sum\limits_{m=4}^{\infty} |\nabla u|_{m-1} \frac{(m-3)
  \tau^{m-4}}{(m-3)!^s} \leq C \tau |u|_{1,\infty}
  \Vert{u}\Vert_{Y_\tau},
\end{align*}
  concluding the proof of \eqref{eq:low1}.

Proof of \eqref{eq:low2}: As in the proof of \eqref{eq:low1}
above, we have
\begin{align}
  \sum\limits_{m=3}^{\infty} \CCC_{m,2} & \leq
  C \sum\limits_{|\alpha|=3,4}
  \sum\limits_{|\beta|=2,\beta\leq\alpha} \tau^{m-3}
  \left( M_\beta \Vert{\partial^\beta u}\Vert_{L^\infty}\right) \left(
  M_{\alpha-\beta} \Vert{\partial^{\alpha-\beta} \nabla
  u}\Vert_{L^2}\right)\notag\\
  & \qquad \qquad \times M_\alpha M_{\beta}^{-1} M_{\alpha-\beta}^{-1} {\alpha \choose \beta} \notag\\
  & + C\sum\limits_{m=5}^{\infty} \sum\limits_{|\alpha|=m}
  \sum\limits_{|\beta|=2, \beta\leq \alpha}
  \left( M_\beta \Vert{\partial^\beta u}\Vert_{L^\infty} \right) \left(M_{\alpha-\beta} \Vert{\partial^{\alpha-\beta} \nabla u}\Vert_{L^2} \frac{(m-4) \tau^{m-5}}{(m-4)!^s}
  \right) \notag\\
  & \qquad \qquad \times  M_\alpha M_{\beta}^{-1} M_{\alpha-\beta}^{-1} {\alpha \choose \beta} \frac{1}{(m-4)
  (m-3)^s} \tau^2. \label{eq:low2-1}
\end{align}
Using Lemma~\ref{lemma:product}, the first sum on the right of
\eqref{eq:low2-1} can be estimated from above by
\begin{align*}
C |u|_{2,\infty} |\nabla u|_1  + C \tau |u|_{2,\infty} |\nabla u|_2
\leq C |u|_{2,\infty} |u|_2 + C \tau |u|_{2,\infty} |u|_3.
\end{align*}
On the other hand, since $s\geq 1$, $|\beta|=2$, and $|\alpha|=m
\geq 5$, we have by Lemma~\ref{lemma:choose} that
\begin{align*}
M_\alpha M_{\beta}^{-1} M_{\alpha-\beta}^{-1} {\alpha \choose \beta}
\frac{1}{(m-4)
  (m-3)^s} \leq {m\choose 2} \frac{1}{(m-4)(m-3)} \leq C.
\end{align*}
By Lemma~\ref{lemma:product}, the second sum on the right of
\eqref{eq:low2-1} is thus bounded by
\begin{align*}
C \tau^2 \sum\limits_{m=5}^{\infty} |u|_{2,\infty} |\nabla u|_{m-2}
\frac{(m-4) \tau^{m-5}}{(m-4)!^s} \leq C \tau^2 |u|_{2,\infty}
\Vert{u}\Vert_{Y_\tau}.
\end{align*}
This proves the desired estimate.

Proof of \eqref{eq:lowmed}: We first observe that the H\"older
inequality and the Sobolev inequality give
\begin{align*}
\Vert{\partial^\beta u \cdot \nabla \partial^{\alpha-\beta}
  u}\Vert_{L^2} \leq C \Vert{\partial^\beta
  u}\Vert_{L^2}^{1/4} \Vert{\Delta \partial^\beta u}\Vert_{L^2}^{3/4} \Vert{\nabla \partial^{\alpha-\beta}
  u}\Vert_{L^2}.
\end{align*}
Therefore we can bound the right hand side of \eqref{eq:lowmed} as
follows
\begin{align*}
& \sum\limits_{m=6}^{\infty} \sum\limits_{j=3}^{[m/2]} \CCC_{m,j}
\leq \sum\limits_{m=6}^{\infty} \sum\limits_{j=3}^{[m/2]}
\sum\limits_{|\alpha|=m}
\sum\limits_{|\beta|=j, \beta\leq \alpha}  \left( M_\beta \Vert{\partial^\beta u}\Vert_{L^2} \frac{\tau^{j-3}}{(j-3)!^s}\right)^{1/4} \tau^{3/2} {\mathcal
A}_{\alpha,\beta,s}  \notag\\
&\qquad \qquad \times \left( M_\beta \Vert{\partial^\beta \Delta u}\Vert_{L^2} \frac{\tau^{j-1}}{(j-1)!^s}\right)^{3/4}  \left( M_{\alpha-\beta}
\Vert{\partial^{\alpha-\beta} \nabla u}\Vert_{L^2} \frac{(m-j-2)
\tau^{m-j-3}}{(m-j-2)!^s}\right)  ,
\end{align*}
where
\begin{align*}
{\mathcal A}_{\alpha,\beta,s} = M_\alpha M_{\beta}^{-1}
M_{\alpha-\beta}^{-1} {\alpha \choose \beta} \frac{(j-3)!^{s/4}
(j-1)!^{3s/4} (m-j-2)!^s}{(m-3)!^s
  (m-j-2)}.
\end{align*} By Lemma~\ref{lemma:choose}, we have that for
$ m \geq 6$ and  $3\leq j \leq [ m/2]$
\begin{align*}
 {\mathcal
A}_{\alpha,\beta,s} & \leq C {m \choose j} {m-3 \choose j-1}^{-s}
\frac{1}{(m-j-2)(j-1)^{s/4}
  (j-2)^{s/4}} \notag\\
  & \leq C {m-3 \choose j-1}^{-s+1} \frac{1}{j^{1 + s/2}}.
\end{align*}
Since $s\geq 1$ the above chain of inequalities gives that ${\mathcal A}_{\alpha,\beta,s}
\leq C$. Together with Lemma~\ref{lemma:product} and the discrete
H\"older inequality this shows that
\begin{align*}
\sum\limits_{m=6}^{\infty} \sum\limits_{j=3}^{[m/2]} \CCC_{m,j} &
\leq C \tau^{3/2}  \sum\limits_{m=6}^{\infty}
\sum\limits_{j=3}^{[m/2]} \left( |u|_{j} \frac{\tau^{j-3}}{(j-3)!^s}
\right)^{1/4} \left( |\Delta u|_{j} \frac{\tau^{j-1}}{(j-1)!^s}
\right)^{3/4} \notag \\
& \qquad \times \left( |\nabla u|_{m-j} \frac{(m-j-2)
\tau^{m-j-3}}{(m-j-2)!^s}\right).
\end{align*}
The discrete Young and H\"older inequalities then give
\begin{align*}
\sum\limits_{m=6}^{\infty} \sum\limits_{j=3}^{[m/2]} \CCC_{m,j} &
\leq C \tau^{3/2} \Vert{u}\Vert_{X_\tau} \Vert{u}\Vert_{Y_\tau},
\end{align*}
concluding the proof of \eqref{eq:lowmed}.

To prove \eqref{eq:highmed}--\eqref{eq:highm} we proceed as in the
proofs of \eqref{eq:low1}--\eqref{eq:lowmed} above, with the roles
of $j$ and $m-j$ reversed. Instead of estimating
$\Vert{\partial^\beta u \cdot \nabla
\partial^{\alpha-\beta} u}\Vert_{L^2}$ with
$\Vert{\partial^\beta u}\Vert_{L^\infty}
\Vert{\partial^{\alpha-\beta} \nabla u}\Vert_{L^2}$ we instead bound
\begin{align*}
\Vert{\partial^\beta u \cdot \nabla
\partial^{\alpha-\beta} u}\Vert_{L^2} \leq \Vert{\partial^\beta u}\Vert_{L^2}
\Vert{\partial^{\alpha-\beta} \nabla u}\Vert_{L^\infty}.
\end{align*}
We omit further details. This concludes the proof of
Lemma~\ref{lemma:commutator}.
\end{proof}

\section{The pressure estimate}\label{sec:pressure}

In the proof of the Lemma~\ref{lemma:pressure} we need to use the
following higher regularity estimate on the solution of the Neumann
problem associated to the Poisson equation for the half-space.
\begin{lemma} \label{lemma:neumann}
Assume that $p$ is a smooth solution of the Neumann problem
\begin{align*}
  - \Delta p &= v\ \mbox{in}\ \Omega,\\
  \frac{\partial p}{\partial n} &= 0\ \mbox{on}\ \partial \Omega ,
\end{align*}
with $v \in C^\infty(\Omega)$. Then there is a universal constant
$C>0$ such that
\begin{align}
  \Vert{\partial_3 \partial^\alpha p}\Vert_{L^2} &\leq C
  \sum\limits_{\substack{s,t \in {\mathbb N}_0,|\beta| = m-1\\ \beta' - \alpha' = (2s,2t)}} {s+t \choose s} \Vert{\partial^\beta
  v}\Vert_{L^2},
  \label{eq:neumannlemma3}
\end{align}
for any $m\geq 1$ and any multiindex $\alpha\in{\mathbb N}_{0}^{3}$
with $|\alpha|=m$ and $\alpha_3 \neq 0$. Additionally, if $\alpha_3
\geq 2$ then
\begin{align}
  \Vert{\partial_1 \partial^\alpha p}\Vert_{L^2} &\leq C
  \sum\limits_{\substack{ s,t \in {\mathbb N}_0,|\beta| = m-1\\ \beta' - \alpha' = (2s+1,2t)}} {s+t \choose s} \Vert{\partial^\beta
  v}\Vert_{L^2},
  \label{eq:neumannlemma1}\\
    \Vert{\partial_2 \partial^\alpha p}\Vert_{L^2} &\leq C
  \sum\limits_{\substack{ s,t \in {\mathbb N}_0, |\beta| = m-1\\ \beta' - \alpha' = (2s,2t+1)}} {s+t \choose s} \Vert{\partial^\beta v}\Vert_{L^2},
  \label{eq:neumannlemma2}
\end{align} where $C>0$ is a universal constant.
\end{lemma}We emphasize that the constant $C$ in the above lemma is independent of $\alpha$ and
$m$. In \eqref{eq:neumannlemma3} we have are summing over the set
\begin{align*}
  \{ \beta \in {\mathbb N}_{0}^3 : |\beta|=m-1,\ \exists s,t \in
  {\mathbb N}_{0}\ \mbox{such that}\ \beta' - \alpha' = (2s,2t)\}
\end{align*}
and similar conventions are used in \eqref{eq:neumannlemma1},
\eqref{eq:neumannlemma2}, and throughout this section.
\begin{proof}
In order to avoid repetition, we only prove \eqref{eq:neumannlemma3}
and indicate the necessary changes for
\eqref{eq:neumannlemma1} and \eqref{eq:neumannlemma2}. Let $\Delta' =
\partial_{11} +
\partial_{22}$ be the tangential Laplacian. Using induction on $k
\in {\mathbb N}_{0}$ we obtain the identity
\begin{align*}
  \partial_{3}^{2k+2} p = (-\Delta')^{k+1}p - \sum\limits_{j=0}^{k} \partial_{3}^{2j}
  (-\Delta')^{k-j} v,
\end{align*}
and upon applying $\partial_3$ to the above equation
\begin{align*}
  \partial_{3}^{2k+3} p = \partial_3(-\Delta')^{k+1}p - \sum\limits_{j=0}^{k} \partial_{3}^{2j+1}
  (-\Delta')^{k-j} v.
\end{align*}
Therefore given $|\alpha|=m$, with $\alpha_3 =2k+1\geq 1$, we have
\begin{align}
  \partial_{3} \partial^\alpha p = \partial_{3}^{2k+2} \partial^{\alpha'} p = (-\Delta')^{k+1}\partial^{\alpha'} p + \sum\limits_{j=0}^{k} (-1)^{k-j+1} \partial_{3}^{2j}
  (\partial_{11} + \partial_{22})^{k-j} \partial^{\alpha'} v,
  \label{eq:p1}
\end{align}
and if $\alpha_3 = 2k+2 \geq 2$, we have
\begin{align}
  \partial_{3} \partial^\alpha p =\partial_{3}^{2k+3} \partial^{\alpha'} p = \partial_3(-\Delta')^{k+1}\partial^{\alpha'} p + \sum\limits_{j=0}^{k} (-1)^{k-j+1} \partial_{3}^{2j+1}
  (\partial_{11} + \partial_{22})^{k-j} \partial^{\alpha'} v.
  \label{eq:p2}
\end{align}
Since $n=(0,0,-1)$, the function $g = (-\Delta')^k
\partial^{\alpha'}p$ satisfies the Neumann problem
\begin{align*}
  -\Delta g = (-\Delta')^k \partial^{\alpha'} v\qquad  \mbox{in}\ \Omega,\\
  \frac{\partial g}{\partial n} = 0\qquad \mbox{on} \ \partial \Omega.
\end{align*}
Using the classical $H^2$-regularity argument for the Neumann
problem we then have
\begin{align*}
  \Vert{\Delta' g}\Vert_{L^2} \leq C
  \Vert{(-\Delta')^k \partial^{\alpha'} v}\Vert_{L^2},
\end{align*}
and
\begin{align*}
  \Vert{\partial_3 \Delta' g}\Vert_{L^2} \leq C
  \Vert{\partial_3 (-\Delta')^k \partial^{\alpha'}
  v}\Vert_{L^2},
\end{align*}
for a positive universal constant $C$. Combining the above estimates
with \eqref{eq:p1}, \eqref{eq:p2}, and the identity
\begin{align*}
  (\partial_{11} + \partial_{22})^m w = \sum\limits_{s=0}^{m} {m
  \choose s} \partial_{1}^{2s} \partial_{2}^{2m-2s} w,
\end{align*}
we obtain
\begin{align}
  \Vert{\partial_3 \partial^\alpha p}\Vert_{L^2} & \leq C
  \sum_{j=0}^{k} \Vert{\partial_{3}^{2j} (\partial_{11} + \partial_{22})^{k-j} \partial^{\alpha'}
  v}\Vert_{L^2} \notag\\
  &\leq C \sum\limits_{j=0}^{k} \sum\limits_{s=0}^{k-j} {k-j \choose s}
  \Vert{\partial_{1}^{2s + \alpha_1} \partial_{2}^{2k-2j-2s+\alpha_2} \partial_{3}^{2j}
  v}\Vert_{L^2} \label{eq:p3}
\end{align}
if $\alpha_3 = 2k+1\geq 1$, and
\begin{align}
  \Vert{\partial_3 \partial^\alpha p}\Vert_{L^2} & \leq C
  \sum_{j=0}^{k} \Vert{\partial_{3}^{2j+1} (\partial_{11} + \partial_{22})^{k-j} \partial^{\alpha'}
  v}\Vert_{L^2} \notag\\
  &\leq C
  \sum\limits_{j=0}^{k} \sum\limits_{s=0}^{k-j} {k-j \choose s}
  \Vert{\partial_{1}^{2s + \alpha_1} \partial_{2}^{2k-2j-2s+\alpha_2} \partial_{3}^{2j+1}
  v}\Vert_{L^2} \label{eq:p44}
\end{align}
if $\alpha_3 = 2k+2\geq 2$. To simplify \eqref{eq:p3} above, let $t
= k-j-s \geq 0$ and $\beta = (2s + \alpha_1,2k-2j-2s+\alpha_2,2j) =
(2s + \alpha_1, 2t + \alpha_2, \alpha_3 - 1 - 2s -2t) \in {\mathbb
N}_{0}^{3}$. Since $|\alpha|=m$ and $\alpha_3 = 2k+1$, we have
$|\beta|=m-1$, and by re-indexing the sums, \eqref{eq:p3} can be
re-written as
\begin{align*}
  \Vert{\partial_3 \partial^\alpha p}\Vert_{L^2} \leq C
  \sum\limits_{\substack{s,t\in{\mathbb N}_0, |\beta|=m-1\\ \beta'-\alpha'=(2s,2t)}} \ {s+t \choose s}
  \Vert{\partial^\beta v}\Vert_{L^2},
\end{align*}
The above estimate also holds for $\alpha_3 = 2k+2$ with the
substitution $\beta = (2s + \alpha_1,2k-2j-2s+\alpha_2,2j+1)$,
thereby simplifying the upper bound \eqref{eq:p44}, and concluding the proof of \eqref{eq:neumannlemma3}.

To prove \eqref{eq:neumannlemma1} we proceed as above and obtain
\begin{align}
\Vert{\partial_1 \partial^\alpha p}\Vert_{L^2}  & =
\Vert{\partial_{1}^{\alpha_1 +1}
\partial_{2}^{\alpha_2} \partial_{3}^{2k+2} p }\Vert_{L^2} \notag\\
& \qquad \qquad \leq C
\sum\limits_{j=0}^{k} \sum\limits_{s=0}^{k-j} {k-j \choose s}
\Vert{\partial_{1}^{\alpha_1 + 2s +1} \partial_{2}^{\alpha_2 + 2k -
2j - 2s} \partial_{3}^{2j} v}\Vert_{L^2} \label{eq:D1even}
\end{align}
if $\alpha_3 = 2k+2 \geq 2$, and
\begin{align}
\Vert{\partial_1 \partial^\alpha p}\Vert_{L^2}  &=
\Vert{\partial_{1}^{\alpha_1 +1}
\partial_{2}^{\alpha_2} \partial_{3}^{2k+3} p }\Vert_{L^2} \notag\\
& \qquad \qquad \leq C
\sum\limits_{j=0}^{k} \sum\limits_{s=0}^{k-j} {k-j \choose s}
\Vert{\partial_{1}^{\alpha_1 + 2s +1} \partial_{2}^{\alpha_2 + 2k -
2j - 2s} \partial_{3}^{2j+1} v}\Vert_{L^2} \label{eq:D1odd}
\end{align}
if $\alpha_3 = 2k+3 \geq 3$. In \eqref{eq:D1even} we let $t = k-j-s
\geq 0$ and $\beta=(\alpha_1 + 2s + 1, \alpha_2 + 2t,2j) = (\alpha_1
+ 2s + 1, \alpha_2 + 2t,\alpha_3 - 2 - 2s - 2t)$, since $\alpha_3 =
2k +2$ and $|\alpha|=m$. Similarly in \eqref{eq:D1odd} we let $\beta
= (\alpha_1 + 2s + 1, \alpha_2 + 2t,2j+1) = (\alpha_1 + 2s + 1,
\alpha_2 + 2t,\alpha_3 - 2 - 2s - 2t)$, since $\alpha_3 = 2k +3$ and
$|\alpha|=m$. The above substitutions and re-indexing prove
\eqref{eq:neumannlemma1}. Upon permuting the first and second
coordinates, this also proves \eqref{eq:neumannlemma2}.
\end{proof}
\begin{remark}\label{rem:pres}
We note that Lemma~\ref{lemma:neumann} does not give an estimate for
$\Vert{\partial_1 \partial^\alpha p}\Vert_{L^2}$ and
$\Vert{\partial_2 \partial^\alpha p}\Vert_{L^2}$ if $\alpha_3=1$. In
this case we note that the function $g = \partial^{\alpha'} p$
satisfies the Neumann problem
\begin{align*}
  -\Delta g = \partial^{\alpha'} v\qquad  \mbox{in}\ \Omega,\\
  \frac{\partial g}{\partial n} = 0\qquad \mbox{on} \ \partial \Omega.
\end{align*} The classical $H^2$-regularity argument then gives
\begin{align*}
\Vert{\partial_1 \partial^\alpha p}\Vert_{L^2} = \Vert{\partial_1
\partial_3
\partial^{\alpha'} p}\Vert_{L^2}\leq C \Vert{\partial^{\alpha'}
v}\Vert_{L^2},
\end{align*}
and
\begin{align*}
\Vert{\partial_2 \partial^\alpha p}\Vert_{L^2} = \Vert{\partial_2
\partial_3
\partial^{\alpha'} p}\Vert_{L^2}\leq C \Vert{\partial^{\alpha'}
v}\Vert_{L^2},
\end{align*}
for a positive universal constant $C>0$.
\end{remark}
We note that Lemma~\ref{lemma:neumann} is different from the
classical higher regularity estimates (cf.~\cite{GT,LM,T}) for the
Neumann problem in the fact that the constant $C$ in
\eqref{eq:neumannlemma3}--\eqref{eq:neumannlemma2} does not increase
with $m$. The dependence on $m$ is encoded in the sums with binomial
weights on the right side of
\eqref{eq:neumannlemma3}--\eqref{eq:neumannlemma2}.

The following lemma shows that only a factor of $m$ is lost in the
above higher regularity estimates if each $\Vert{\partial^\beta
v}\Vert_{L^2}$ term is paired with a proper binomial weight. This
explains the definition of the homogeneous Sobolev norms $|\cdot|_m$
in \eqref{eq:m}.
\begin{lemma}\label{lemma:*}
  There exists a positive universal constant $C$ such that
  \begin{align}
    \sum\limits_{s=0}^{[\beta_1/2]} \sum\limits_{t=0}^{[\beta_2/2]} {
    \beta_1 + \beta_2 - 2s - 2t \choose \beta_1 - 2s} {s+t \choose
    s} &\leq C m {\beta_1 + \beta_2 \choose \beta_1} \label{*}
  \end{align}
for any $m\geq 3$ and any multi-index $\beta =
(\beta_1,\beta_2,m-1-\beta_1-\beta_2) \in
  {\mathbb N}_{0}^{3}$. Additionally, if $\beta_1 \geq 1$ we have
  \begin{align}
    \sum\limits_{s=0}^{[(\beta_1-1)/2]} \sum\limits_{t=0}^{[\beta_2/2]} {
    \beta_1 + \beta_2 - 2s -1 - 2t \choose \beta_1 - 2s -1} {s+t \choose
    s} &\leq C m {\beta_1 + \beta_2 \choose \beta_1}, \label{*1}
  \end{align}
  while if $\beta_2 \geq 1$ we have
  \begin{align}
    \sum\limits_{s=0}^{[\beta_1/2]} \sum\limits_{t=0}^{[(\beta_2-1)/2]} {
    \beta_1 + \beta_2 - 2s - 2t-1 \choose \beta_1 - 2s} {s+t \choose
    s} &\leq C m {\beta_1 + \beta_2 \choose \beta_1}, \label{*2}
  \end{align}
  where $C$ is a universal constant.
\end{lemma}  We note that in particular the constant $C$ is independent of $m$ and $\beta$.
\begin{proof}
Due to symmetry we only give the proof of \eqref{*}. Estimates
\eqref{*1} and \eqref{*2} are proven {\it mutatis-mutandi}. First we
recall that given $\alpha,\gamma \in {\mathbb N}_{0}^{3}$, with
$\gamma \leq \alpha$, we have
\begin{align*}
  {\alpha \choose \gamma } \leq {|\alpha| \choose |\gamma|}.
\end{align*}
Using the above inequality we get
\begin{align*}
  &\sum\limits_{s=0}^{[\beta_1/2]} \sum\limits_{t=0}^{[\beta_2/2]} {
    \beta_1 + \beta_2 - 2s - 2t \choose \beta_1 - 2s} {s+t \choose
    s} {\beta_1 + \beta_2 \choose \beta_1}^{-1} \notag\\
  &\qquad \qquad \qquad \leq   \sum\limits_{s=0}^{[\beta_1/2]} \sum\limits_{t=0}^{[\beta_2/2]} { \beta_1 + \beta_2 - s - t \choose \beta_1 - s} {\beta_1 + \beta_2 \choose \beta_1}^{-1} \leq   \sum\limits_{s=0}^{[\beta_1/2]} \sum\limits_{t=0}^{[\beta_2/2]}  {s+t \choose s}^{-1}.
\end{align*}
The lemma is then proven if we find a constant $C$ such that
\begin{align*}
\sum\limits_{s=0}^{[\beta_1/2]} \sum\limits_{t=0}^{[\beta_2/2]}
{s+t \choose s}^{-1} \leq C (\beta_1 + \beta_2).
\end{align*}
Without loss of generality we may assume that $\beta_1,\beta_2 \geq
4$. We split the above sum into
\begin{align}
  \sum\limits_{s=0}^{[\beta_1/2]} \sum\limits_{t=0}^{[\beta_2/2]}  {s+t \choose s}^{-1} &\leq \sum\limits_{t=0}^{[\beta_2/2]}  {t \choose 0}^{-1}+{t+1\choose 1}^{-1}  + \sum\limits_{s=0}^{[\beta_1/2]}  {s \choose s}^{-1} + {s+1 \choose s}^{-1}     \notag \\
& \qquad + \sum\limits_{s=2}^{[\beta_1/2]} \sum\limits_{t=2}^{[\beta_2/2]}  {s+t \choose s}^{-1} \notag \\
& = T_1 + T_2 + T_3. \label{3}
\end{align}
It is clear that
\begin{align}\label{eq:testtest}
  T_1 + T_2 \leq C (\beta_1 + \beta_2).
\end{align}
We estimate $T_3$ by appealing to the Stirling estimate
(cf.~\cite[p.~200]{R})
\begin{align*}
  e^{7/8} \sqrt{n} \left(\frac ne\right)^n < n! <e \sqrt{n} \left(\frac
  ne\right)^n.
\end{align*}
This implies
\begin{align*}
  \frac{s! t!}{(s+t)!} \leq e^{9/8} \sqrt{\frac{st}{s+t}} \frac{1}{(1 + s/t)^t} \frac{1}{(1+t/s)^s}.
\end{align*}
Thus we obtain
\begin{align}
  T_3 &\leq C \sum\limits_{s=2}^{[\beta_1/2]} \sum\limits_{t=2}^{[\beta_2/2]} \sqrt{t} \frac{1}{(1+t/s)^s}. \label{9}
\end{align}
Since $s\geq 2$, the Binomial Theorem implies
\begin{align*}
\left(1+\frac ts\right)^s \geq 1 + {s \choose 2}
\left(\frac{t}{s}\right)^2,
\end{align*}
and by \eqref{9} we have
\begin{align*}
  T_3 &\leq C \sum\limits_{s=2}^{[\beta_1/2]} \sum\limits_{t=2}^{[\beta_2/2]} \sqrt{t} \frac{1}{t^2} \leq C \left(\sum\limits_{s=2}^{[\beta_1/2]} 1\right) \left( \sum\limits_{t=2}^{\infty} \frac{1}{t^{3/2}} \right) \leq C \beta_1.
\end{align*}
Since $\beta_1
+\beta_2 \leq m-1$, the above inequality, \eqref{3}, and \eqref{eq:testtest} complete the proof of the lemma.\end{proof}

\begin{proof}[Proof of Lemma~\ref{lemma:pressure}]
First, note that since $p$ satisfies the elliptic Neumann problem \eqref{eq:P1}--\eqref{eq:P20} we
may use Lemma~\ref{lemma:neumann} to estimate higher derivatives of
$\partial_3 p$ as
\begin{align*}
  \sum\limits_{|\alpha|=m, \alpha_3\neq
  0} M_\alpha \Vert{\partial_3 \partial^\alpha
  p}\Vert_{L^2} \leq C
  \sum\limits_{|\alpha|=m, \alpha_3\neq 0}  \sum\limits_{\substack{s,t\in {\mathbb N}_0,|\beta|=m-1\\ \beta'-\alpha' = (2s,2t)}} M_\alpha
  {s+t \choose s} \Vert{\partial^\beta
  (\partial_i u_k \partial_k u_i)}\Vert_{L^2}.
\end{align*}
By re-indexing the terms in the parenthesis, the right side of the
above inequality may be re-written as
\begin{align*}
\sum\limits_{|\beta|=m-1} \sum\limits_{s=0}^{[\beta_1/2]}
\sum\limits_{t=0}^{[\beta_2/2]}
  {\beta_1 + \beta_2 -2s -2t \choose \beta_1 -2s}
  {s+t \choose s} \Vert{\partial^\beta
  (\partial_i u_k \partial_k u_i)}\Vert_{L^2}.
\end{align*}
Using the estimate \eqref{*} of Lemma~\ref{lemma:*} we bound the above
expression by
\begin{align*}
  C m \sum\limits_{|\beta|=m-1} M_\beta
  \Vert{\partial^\beta (\partial_i u_k \partial_k u_i)}\Vert_{L^2}
\end{align*}
and therefore
\begin{align}\label{eq:p3est}
  & \sum\limits_{m=3}^{\infty} \left(
  \sum\limits_{|\alpha|=m,\alpha_3\neq 0} M_\alpha \Vert{\partial_3 \partial^\alpha
  p}\Vert_{L^2}\right) \frac{\tau^{m-3}}{(m-3)!^s} \notag\\
  & \qquad \qquad \qquad \leq C
  \sum\limits_{m=3}^{\infty}  \sum\limits_{|\beta|=m-1} M_\beta
  \Vert{\partial^\beta (\partial_i u_k \partial_k u_i)}\Vert_{L^2} \frac{m
  \tau^{m-3}}{(m-3)!^s}.
\end{align} On the other hand, higher derivatives of $\partial_1 p$ are
estimated using the decomposition
\begin{align}
 \sum\limits_{|\alpha|=m,
\alpha_3\neq 0} M_\alpha \Vert{\partial_1
\partial^\alpha p}\Vert_{L^2} & =  \sum\limits_{|\alpha|=m,
\alpha_3=1} M_\alpha \Vert{\partial_1
\partial^\alpha p}\Vert_{L^2} + \sum\limits_{|\alpha|=m,
\alpha_3\geq 2} M_\alpha \Vert{\partial_1
\partial^\alpha p}\Vert_{L^2}. \label{eq:ppp1}
\end{align}
By Remark~\ref{rem:pres}, the first term on the right of
\eqref{eq:ppp1} is bounded by
\begin{align}
  C \sum\limits_{|\alpha|=m,
\alpha_3=1} M_\alpha \Vert{
\partial^{\alpha'} \left(\partial_i u_k \partial_k u_i\right) }\Vert_{L^2}
=  C \sum\limits_{|\beta|=m-1, \beta_3 =  0} M_\beta \Vert{
\partial^\beta \left(\partial_i u_k \partial_k u_i\right)}\Vert_{L^2}.
\label{eq:ppp2}
\end{align}
Using estimate \eqref{eq:neumannlemma1}, the second term on the
right side of \eqref{eq:ppp1} is estimated by
\begin{align*}
  C \sum\limits_{|\alpha|=m,\alpha_3 \geq 2} \left( \sum\limits_{\substack{s,t\in {\mathbb N}_{0},|\beta|=m-1\\ \beta' - \alpha' = (2s+1,2t)}}
  M_\alpha {s+t \choose s} \Vert{\partial^\beta \left( \partial_i u_k \partial_k u_i\right)}\Vert_{L^2}\right).
\end{align*}
By re-indexing the above expression equals
\begin{align*}
  C \sum\limits_{|\beta|=m-1, \beta_1 \geq 1}
  \sum\limits_{s=0}^{[(\beta-1)/2]} \sum\limits_{t=0}^{[\beta_2 /2]}
  {\beta_1 -1 + \beta_2 - 2s -2t \choose \beta_1 - 2s - 1} {s+t
  \choose s} \Vert{\partial^\beta \left(\partial_i u_k \partial_k u_i\right)}\Vert_{L^2},
\end{align*}
and using \eqref{*1} it is bounded from above by
\begin{align}
  C m \sum\limits_{|\beta|=m-1, \beta_1 \geq 1} M_\beta \Vert{\partial^\beta \left(\partial_i u_k \partial_k u_i\right)}\Vert_{L^2}. \label{eq:ppp3}
\end{align}
Therefore, by \eqref{eq:ppp1}, \eqref{eq:ppp2}, and \eqref{eq:ppp3},
we have
\begin{align}\label{eq:p1est}
  &\sum\limits_{m=3}^{\infty} \left(
  \sum\limits_{|\alpha|=m,\alpha_3\neq 0} M_\alpha \Vert{\partial_1 \partial^\alpha
  p}\Vert_{L^2}\right) \frac{\tau^{m-3}}{(m-3)!^s} \notag\\
  & \qquad \qquad \qquad \leq C
  \sum\limits_{m=3}^{\infty}  \sum\limits_{|\beta|=m-1} M_\beta
  \Vert{\partial^\beta (\partial_i u_k \partial_k u_i)}\Vert_{L^2} \frac{m
  \tau^{m-3}}{(m-3)!^s}.
\end{align}
By symmetry, we also get
\begin{align} \label{eq:p2est}
  & \sum\limits_{m=3}^{\infty} \left( \sum\limits_{|\alpha|=m, \alpha_3\neq
  0} M_\alpha \Vert{\partial_2 \partial^\alpha
  p}\Vert_{L^2}\right) \frac{\tau^{m-3}}{(m-3)!^s} \notag\\
  & \qquad \qquad \qquad \leq C \sum\limits_{m=3}^{\infty} \sum\limits_{|\beta|=m-1} M_\beta
  \Vert{\partial^\beta (\partial_i u_k \partial_k u_i)}\Vert_{L^2} \frac{m
  \tau^{m-3}}{(m-3)!^s}.
\end{align}
Combining \eqref{eq:p3est}, \eqref{eq:p1est}, \eqref{eq:p2est}, and
the Leibniz rule we obtain
\begin{align}
  \PPP \leq C \sum\limits_{m=3}^{\infty} \sum\limits_{|\beta|=m-1} M_\beta
  \Vert{\partial^\beta (\partial_i u_k \partial_k u_i)}\Vert_{L^2} \frac{m
  \tau^{m-3}}{(m-3)!^s}\leq C \sum\limits_{m=3}^{\infty} \sum\limits_{j=0}^{m-1}
  \PPP_{m,j}, \label{eq:Pest}
\end{align} where
\begin{align*}
  \PPP_{m,j} = \frac{m \tau^{m-3}}{(m-3)!^s} \sum\limits_{|\beta|=m-1}
\sum\limits_{|\gamma| = j, \gamma \leq \beta}
  M_\beta {\beta \choose \gamma}
  \Vert{  \partial^\gamma \partial_i u_k \cdot \partial^{\beta-\gamma}
  \partial_k u_i}\Vert_{L^2}.
\end{align*}We split the right side of \eqref{eq:Pest} into seven terms according to
the values of $m$ and $j$. For low $j$, we claim
\begin{align}
  & \sum\limits_{m=3}^{\infty} \PPP_{m,0} \leq C |u|_{1,\infty} |u|_3 + C\tau |u|_{1,\infty} \Vert{u}\Vert_{Y_\tau}\label{eq:Plow0}\\
  & \sum\limits_{m=3}^{\infty} \PPP_{m,1} \leq C |u|_{2,\infty} |u|_2 + C\tau |u|_{2,\infty} |u|_3 + C \tau^2 |u|_{2,\infty} \Vert{u}\Vert_{Y_\tau}\label{eq:Plow1}\\
  & \sum\limits_{m=5}^{\infty} \PPP_{m,2} \leq  C\tau^2 |u|_{3,\infty} |u|_3 + C \tau^3 |u|_{3,\infty} \Vert{u}\Vert_{Y_\tau}\label{eq:Plow2}
  \end{align} for intermediate $j$, we have
  \begin{align}
  & \sum\limits_{m=8}^{\infty} \sum\limits_{j=3}^{[m/2]-1} \PPP_{m,j} \leq C \tau^{3/2} \Vert{u}\Vert_{X_\tau} \Vert{u}\Vert_{Y_\tau}\label{eq:Plowmed}\\
  & \sum\limits_{m=6}^{\infty} \sum\limits_{j=[m/2]}^{m-3} \PPP_{m,j} \leq C \tau^{3/2} \Vert{u}\Vert_{X_\tau} \Vert{u}\Vert_{Y_\tau} \label{eq:Phighmed}
  \end{align} and for high $j$, we claim
  \begin{align}
  & \sum\limits_{m=4}^{\infty} \PPP_{m,m-2} \leq C\tau |u|_{2,\infty} |u|_3 + C \tau^2 |u|_{2,\infty} \Vert{u}\Vert_{Y_\tau}\label{eq:Phighm-2}\\
  & \sum\limits_{m=3}^{\infty} \PPP_{m,m-1} \leq C |u|_{1,\infty} |u|_3 + C \tau |u|_{1,\infty} \Vert{u}\Vert_{Y_\tau}\label{eq:Phighm-1}.
\end{align}
The above estimates are proven similarly to
\eqref{eq:low1}--\eqref{eq:highm} in the proof of
Lemma~\ref{lemma:commutator}. Due to symmetry we have
presented there the proofs of the estimates where $j\leq m-j$. For
completeness of the exposition we provide the proofs of
\eqref{eq:Phighmed}--\eqref{eq:Phighm-1}, where we have $m-j
< j$.

Proof of \eqref{eq:Phighmed}: We proceed as in the proof of
\eqref{eq:lowmed} in Section~\ref{sec:commutator}. First, the
H\"older and Sobolev inequalities imply that
\begin{align*}
  \Vert{\partial^\gamma \partial_i u_k \cdot \partial^{\beta-\gamma}
  \partial_k  u_i}\Vert_{L^2} \leq C \Vert{\partial^\gamma \partial_i
  u_k}\Vert_{L^2} \Vert{ \partial^{\beta-\gamma}
  \partial_k u_i}\Vert_{L^2}^{1/4} \Vert{\Delta \partial^{\beta-\gamma}
  \partial_k u_i}\Vert_{L^2}^{3/4}.
\end{align*}
Therefore,
\begin{align*}
  \sum\limits_{m=6}^{\infty} \sum\limits_{j = [m/2]}^{m-3}
  \PPP_{m,j} & \leq C \sum\limits_{m=6}^{\infty} \sum\limits_{j =
  [m/2]}^{m-3} \sum\limits_{|\beta|=m-1} \sum\limits_{|\gamma|=j, \gamma \leq \beta} \left( M_\gamma \Vert{ \partial^\gamma \partial_i u_k}\Vert_{L^2} \frac{(j-2) \tau^{j-3}}{(j-2)!^s}\right) \notag \\
  &\qquad \times  \left( M_{\beta-\gamma} \Vert{ \partial^{\beta-\gamma} \partial_k u_i}\Vert_{L^2} \frac{\tau^{m-j-3}}{(m-j-3)!^s}\right)^{1/4}  \notag\\
  & \qquad \times \left( M_{\beta-\gamma} \Vert{\Delta \partial^{\beta-\gamma} \partial_k u_i}\Vert_{L^2} \frac{\tau^{m-j-1}}{(m-j-1)!^s}\right)^{3/4} \tau^{3/2} {\mathcal
  B}_{\beta,\gamma,s},
\end{align*}
where
\begin{align*}
{\mathcal B}_{\beta,\gamma,s} = M_\beta M_{\gamma}^{-1}
  M_{\beta-\gamma}^{-1} {\beta \choose \gamma} \frac{m (j-2)!^s (m-j-3)!^{s/4}
  (m-j-1)!^{3s/4}}{(j-2)(m-3)!^s}.
\end{align*}
By Lemma~\ref{lemma:choose} we have that for $m\geq 6$ and
$[m/2]\leq j \leq m-3$
\begin{align*}
{\mathcal B}_{\beta,\gamma,s} &\leq C {m-1 \choose j} {m-3 \choose
j-2}^{-s} \frac{m}{(j-2) (m-j-1)^{s/4} (m-j-2)^{s/4}}\\
& \leq C {m-3 \choose j-2}^{1-s} (m-j)^{-s/2},
\end{align*}since ${m-1 \choose j} \leq C {m-3 \choose j-2}$, when $j \geq m/2$.
Therefore, ${\mathcal B}_{\beta,\gamma,s} \leq C$; hence, by
Lemma~\ref{lemma:product} and the discrete H\"older inequality, we
have
\begin{align*}
\sum\limits_{m=6}^{\infty} \sum\limits_{j = [m/2]}^{m-3} \PPP_{m,j}
&\leq C \tau^{3/2} \sum\limits_{m=6}^{\infty}
\sum\limits_{j=[m/2]}^{m-3} \left( |\partial_k
u_i|_{m-j-1}
\frac{\tau^{m-j-3}}{(m-j-3)!^s}\right)^{1/4}\\
& \qquad \times \left( |\Delta \partial_k u_i|_{m-j-1}
\frac{\tau^{m-j-1}}{(m-j-1)!^s} \right)^{3/4} \left( |\partial_i u_k|_{j}
\frac{(j-2)\tau^{j-3}}{(j-2)!^s}\right) .
\end{align*}
The discrete Young and H\"older inequalities then give
\begin{align*}
  \sum\limits_{m=6}^{\infty} \sum\limits_{j = [m/2]}^{m-3}
\PPP_{m,j}\leq C \tau^{3/2}\Vert{u}\Vert_{X_\tau}
\Vert{u}\Vert_{Y_\tau},
\end{align*}
concluding the proof of \eqref{eq:Phighmed}.

Proof of \eqref{eq:Phighm-2}: As above we use the H\"older
inequality and obtain
\begin{align*}
  \sum\limits_{m=4}^{\infty} \PPP_{m,m-2} & \leq
  \sum\limits_{m=4}^{\infty} \sum\limits_{|\beta|=m-1}
  \sum\limits_{|\gamma|=m-2,\gamma\leq \beta} M_\beta {\beta \choose
  \gamma} \Vert{\partial^\gamma \partial_i u_k}\Vert_{L^2}
  \Vert{\partial^{\beta-\gamma} \partial_k u_i}\Vert_{L^\infty}
  \frac{m \tau^{m-3}}{(m-3)!^s} \\
  &\leq C\tau \sum\limits_{|\beta|=3} \sum\limits_{|\gamma|=2,
  \gamma\leq\beta} \Vert{\partial^\gamma
  \partial_i u_k}\Vert_{L^2}
  \Vert{\partial^{\beta-\gamma}
  \partial_k u_i}\Vert_{L^\infty} \\
  & + C \tau^2 \sum\limits_{m=5}^{\infty} \sum\limits_{|\beta|=m-1}
  \sum\limits_{|\gamma|=m-2,\gamma\leq\beta} \left(M_\gamma
  \Vert{\partial^\gamma \partial_i u_k}\Vert_{L^2} \frac{m
  \tau^{m-5}}{(m-4)!^s}\right)  \\
  & \qquad \times \left(M_{\beta-\gamma}
  \Vert{\partial^{\beta-\gamma}
  \partial_k u_i}\Vert_{L^\infty}\right) M_\beta M_{\gamma}^{-1}
  M_{\beta-\gamma}^{-1} {\beta \choose \gamma} \frac{1}{(m-3)^s}.
\end{align*}
Using Lemma~\ref{lemma:choose}, Lemma~\ref{lemma:product}, and
$s\geq 1$, this shows that the far right side of the above chain of
inequalities is bounded by
\begin{align*}
  C\tau |\partial_i u_k|_2 |\partial_k u_i|_{1,\infty} &+ C \tau^2 |\partial_k u_i|_{1,\infty} \sum\limits_{m=5}^{\infty} |
  \partial_i u_k|_{m-2} \frac{m \tau^{m-5}}{(m-4)!^s} \notag\\
  & \qquad \qquad \qquad \leq C\tau |u|_{2,\infty} |u|_3 + C \tau^2
  |u|_{2,\infty} \Vert{u}\Vert_{Y_\tau},
\end{align*} thereby proving \eqref{eq:Phighm-2}.

Proof of \eqref{eq:Phighm-1}: By the H\"older inequality we have
\begin{align*}
  \sum\limits_{m=3}^{\infty} \PPP_{m,m-1} &\leq
  \sum\limits_{m=3}^{\infty} \sum\limits_{|\beta| = m-1} M_\beta
  \Vert{ \partial^\beta \partial_i u_k}\Vert_{L^2}
  \Vert{\partial_k u_i}\Vert_{L^\infty} \frac{m
  \tau^{m-3}}{(m-3)!^s}\\
  &\leq C |\partial_i u_k |_{2} \left \Vert\partial_k u_i\right\Vert_{L^\infty}  + C \tau \left \Vert\partial_k u_i\right\Vert_{L^\infty}
  \sum\limits_{m=4}^{\infty} |\partial_i u_k|_{m-1} \frac{m
  \tau^{m-4}}{(m-3)!^s}\\
  & \leq C |u|_{1,\infty} |u|_3  + C \tau |u|_{1,\infty}
  \Vert{u}\Vert_{Y_\tau},
\end{align*}
which gives the desired estimate. By symmetry, we may similarly
prove \eqref{eq:Plow0}--\eqref{eq:Plowmed}, but in these cases we
apply the H\"older inequality as
\begin{align*}
\Vert{
\partial^\gamma \partial_i u_k \cdot \partial^{\beta-\gamma}
\partial_k u_i}\Vert_{L^2} \leq \Vert{
\partial^\gamma \partial_i u_k}\Vert_{L^\infty} \Vert{ \partial^{\beta-\gamma}
\partial_k u_i}\Vert_{L^2},
\end{align*} that is we reverse the roles of $j$ and $m-j$. We omit further details. This concludes the proof of Lemma~\ref{lemma:pressure}.
\end{proof}

\section*{Acknowledgments} Both authors were supported in part by the NSF grant DMS-0604886.

% You may incorporate your references as follows in your main tex file.
% Using BibTex is not recommended but can be handled.

\end{document}